\documentclass[a4,12pt]{article}
\usepackage{amsmath,amsthm,amssymb,color}

\newcommand{\C}{\mathbb{C}}
\newcommand{\ru}{\rho}                                              
\newcommand{\n}[1]{\left\vert {#1} \right\vert}                    
\newcommand{\N}[1]{\left\Vert {#1} \right\Vert}                    
\newcommand{\inner}[2]{\left\langle {#1} , {#2} \right\rangle}     
\newcommand{\Vol}[1]{{\rm Vol } \left( {#1} \right)}               
     
\newcommand{\cent}{t}                
\newcommand{\Kr}{M_{\ru}}

\begin{document}

 \newtheorem{thm}{Theorem}[section]
 \newtheorem{cor}[thm]{Corollary}
 \newtheorem{lem}[thm]{Lemma}
 \newtheorem{prop}[thm]{Proposition}
 \theoremstyle{definition}
 \newtheorem{defn}[thm]{Definition}
 \theoremstyle{remark}
 \newtheorem{rem}[thm]{Remark}
 \newtheorem*{ex}{Example}
 \numberwithin{equation}{section}

\title{Positive Toeplitz operators on the Bergman space of a minimal bounded homogeneous domain} 
\author{Satoshi Yamaji}
\date{}

%%% ----------------------------------------------------------------------
\maketitle
%%% ----------------------------------------------------------------------

{\bf Abstract} : 
Necessary and sufficient conditions for positive Toeplitz operators 
on the Bergman space of a minimal bounded homogeneous domain 
to be bounded or compact are described in terms of the Berezin transform, the averaging function and the Carleson property. 



%%%%%%%%%%%%%%%%%%%%%%%%%%%%%%%%%%%%%%%%%%%%%%%%%%%%%%%%%%%%%%%%%%%%%%%%%%%%%%%%%%%%%%%%%%%%%%%%%%%%%%%%%%%%%%%%%%%%%%%%%%
\section{Introduction}
In 1988, Zhu obtained the conditions in order that a positive Toeplitz operator is bounded or compact on 
the Bergman space of a bounded symmetric domain in its Harish-Chandra realization \cite{Zhu1}. 
In this paper, we extend this result for the case that the domain is a minimal bounded homogeneous domain. 

Let $D$ be a bounded homogeneous domain in $\C^n$, $dV(z)$ the Lebesgue measure, 
$\mathcal{O}(D)$ the space of all holomorphic functions on $D$, 
and $L^2_a(D)$ the Bergman space $L^2(D,dV) \cap \mathcal{O}(D)$ of $D$. 
We denote by $K_{D}$ the Bergman kernel of $D$, that is, the reproducing kernel of $L^2_a(D)$. 
It is known that $\mathcal{U}$ is a minimal domain with a center $\cent$ if and only if 
$K_{\mathcal{U}}(z,\cent) = K_{\mathcal{U}}(\cent,\cent)$ for any $z \in \mathcal{U} $ (see \cite[ Theorem 3.1]{MM}). 
For example, the open unit disk $\mathbb{D}$, the open unit ball $\mathbb{B}^n$ and the bidisk $\mathbb{D} \times \mathbb{D}$ are minimal domains. 
It is known that every bounded homogeneous domain is biholomorphic to a minimal bounded homogeneous domain (see \cite{I-Y}). 

Let $\mu$ be a complex Borel measure on $\mathcal{U}$. 
The Toeplitz operator $T_{\mu}$ with symbol $\mu$ is defined by 
\begin{eqnarray*}
T_{\mu} f(z) := \int_{\mathcal{U}} K_{\mathcal{U}}(z,w) f(w) \, d\mu(w) \ \ \ (z \in \mathcal{U}) .
\end{eqnarray*}
If $d\mu(w) = u(w) dV(w)$ holds for some $u \in L^{\infty}(\mathcal{U})$, we have $T_{\mu} f =P(uf)$, 
where $P$ is the orthogonal projection from $L^2(\mathcal{U})$ onto $L^2_a(\mathcal{U})$. 
Therefore, $T_{\mu}$ is a bounded operator on $L^2_a(\mathcal{U})$ with $\N{T_{\mu}} \leq \N{u}_{\infty}$. 
We consider the condition of $\mu$ that $T_{\mu}$ is a bounded (or compact) operator on $L^2_a(\mathcal{U})$. 

A Toeplitz operator is called positive if its symbol is positive. 
A result on positive Toeplitz operator of a bounded symmetric domain was obtained in \cite{Zhu1}. 
Zhu proved that the boundedness of the positive Toeplitz operator on $L^2_a(\Omega)$ is equivalent to the boundedness of the 
Berezin transform $\widetilde{\mu}$ or the averaging function $\widehat{\mu}$ on $\Omega$. 
The key lemma is \cite[Lemma 8]{BCZ}. 
The proof of this lemma is based on some characteristic properties of a bounded symmetric domain in its Harish-Chandra realization. 
It is difficult to generalize directly their argument for a bounded homogeneous domain, which is not necessarily symmetric. 
However, the following theorem enables us to prove the same key estimate (Lemma \ref{Lemma1}) for the Bergman kernel of a minimal bounded homogeneous domain.

%%%%%%%%%%%%%%%%%%%%%%%%%%%%%%%%%%%%%%%%%%%%%%%%%%%%%%%%%%%%%%%%%%%%%%%%%%
\begin{thm}[\mbox{\cite[Theorem 1.1]{I-Y}}] \label{Intro1} 
Let $\mathcal{U} \subset \C^n$ be a minimal bounded homogeneous domain. 
Take any $\ru>0$. Then, there exists $C_{\ru}>0$ such that 
$$  C_{\ru}^{-1} \leq \n{\frac{K_{\mathcal{U}}(z,a)}{K_{\mathcal{U}}(a,a)} } \leq C_{\ru} $$
for all $z,a \in \mathcal{U}$ with $\beta (z,a) \leq \ru$, 
where $\beta$ denotes the Bergman distance on $\mathcal{U}$. 
\end{thm}
%%%%%%%%%%%%%%%%%%%%%%%%%%%%%%%%%%%%%%%%%%%%%%%%%%%%%%%%%%%%%%%%%%%%%%%%%%
Using Lemma $\ref{Lemma1}$ and Zhu's method (see \cite{Zhu1} or \cite{Zhu3}), 
we deduce a certain relation of averaging functions to the Carleson measures (Theorem \ref{thmCarleson}). 
Moreover, we obtain the following theorem. 
%%%%%%%%%%%%%%%%%%%%%%%%%%%%%%%%%%%%%%%%%%%%%%%%%%%%%%%%%%%%%%%%%%%%%%%%%%
\begin{thm} \label{Bddness*} 
Let $\mathcal{U} \subset \C^n$ be a minimal bounded homogeneous domain 
and $\mu$ a positive Borel measure on $\mathcal{U}$. 
Then the following conditions are all equivalent.\\
$(a)$ \ $T_{\mu}$ is a bounded operator on $L^2_a(\mathcal{U})$.\\
$(b)$ \ The Berezin transform $\widetilde{\mu}(z) $ is a bounded function on $\mathcal{U}$.\\
$(c)$ \ For all $p \geq 1$, $\mu$ is a Carleson measure for $L^p_a(\mathcal{U})$. \\
$(d)$ \ The averaging function $\widehat{\mu}(z) $ is bounded on $\mathcal{U}$.
\end{thm}
%%%%%%%%%%%%%%%%%%%%%%%%%%%%%%%%%%%%%%%%%%%%%%%%%%%%%%%%%%%%%%%%%%%%%%%%%%

The representative domain of the tube domain over the Vinberg's cone is an example of nonsymmetric minimal bounded homogeneous domain. 
Theorem \ref{Bddness*} generalizes Zhu's result (\cite[Theorem A]{Zhu1}) to such domain, for instance. 

In the part $(c) \Longrightarrow (a)$, 
we use the boundedness of the positive Bergman operator $P_{\mathcal{U}}^+$ 
on $L^2(\mathcal{U},dV)$. 
Using Schur's theorem (see \cite[Theorem 3.6]{Zhu3}), it is sufficient to find 
a positive function $h$ and a positive constant $C$ such that
\begin{eqnarray*}  
\int_{\mathcal{U}} \n{K_{\mathcal{U}}(z,w)} h(w) \, dV(w) \leq C h(z) 
\end{eqnarray*}
holds for all $z \in \mathcal{U}$. 
If $\mathcal{U}$ is a bounded symmetric domain in its Harish-Chandra realization, we can construct such $h$ and $C$ from the Forelli-Rudin inequalities 
(see \cite[Theorem 7.5]{Zhu3}, \cite[Proposition 8]{Eng}). 
But it is difficult to do this on minimal bounded homogeneous domains. 
Instead, we make use of the boundedness of the positive Bergman operator $P_{\mathcal{D}}^+$ on $L^2(\mathcal{D},dV)$, 
where $\mathcal{D}$ is a homogeneous Siegel domain of type II (\cite[Theorem II.7]{B-K}). 
Since every bounded homogeneous domain is biholomorphic to some Siegel domain, we deduce the boundedness of $P_{\mathcal{U}}^+$ (see section 2.4). 

To prove the compactness of $T_{\mu}$, we consider a vanishing Carleson measure for $L^2_a(\mathcal{U})$. 
We know that $K_{\mathcal{U}}(a,a) \rightarrow \infty$ as $a \rightarrow \partial \mathcal{U}$ (see \cite[Proposition 5.2]{Koba}). 
Therefore, we can prove Theorem $\ref{claim:7.2}$ in the same way as in \cite[Theorem 7.7]{Zhu3}. 
We obtain the condition of the compactness of the Toeplitz operator.
%%%%%%%%%%%%%%%%%%%%%%%%%%%%%%%%%%%%%%%%%%%%%%%%%%%%%%%%%%%%%%%%%%%%%%%%%%
\begin{thm} \label{Intro2}
Let $\mathcal{U} \subset \C^n$ be a minimal bounded homogeneous domain 
and $\mu$ a finite positive Borel measure on $\mathcal{U}$. 
Then the following conditions are all equivalent.\\
$(a)$ \ $T_{\mu}$ is a compact operator on $L^2_a(\mathcal{U})$.\\
$(b)$ \ The Berezin transform $\widetilde{\mu}(z)$ tends to $ 0$ as $z \rightarrow \partial \mathcal{U}$.\\ 
$(c)$ \ $\mu$ is a vanishing Carleson measure for $L^2_a(\mathcal{U})$. \\
$(d)$ \ The averaging function $\widehat{\mu}(z)$ tends to $0$ as $z \rightarrow \partial \mathcal{U}$. 
\end{thm}
%%%%%%%%%%%%%%%%%%%%%%%%%%%%%%%%%%%%%%%%%%%%%%%%%%%%%%%%%%%%%%%%%%%%%%%%%%

\section{Preliminaries}
\subsection{Minimal domain}
Let $D$ be a bounded domain in $\C^{n}$. 
We say that $D$ is a minimal domain with a center $t \in D$ if the following condition is satisfied: 
for every biholomorphism $\psi : D \longrightarrow D^{\prime}$ with $\det J(\psi , t)=1$, we have 
\begin{eqnarray*}
{\rm Vol} \, (D^{\prime}) \geq {\rm Vol} \, (D)  .
\end{eqnarray*}
From \cite[Proposition 3.6]{I-K} or \cite[Theorem 3.1]{MM}, 
we see that $D$ is a minimal domain with a center $t$ if and only if 
$$K_{D}(z,t) = \frac{1}{{\rm Vol \, }(D)}$$
for any $z \in D$. 

The representative bounded homogeneous domain is a generalization of the Harish-Chandra 
realization for a bounded symmetric domain. 
Indeed, every bounded homogeneous domain is biholomorphic to a representative bounded homogeneous domain. 
It is known that any representative bounded homogeneous domain is a minimal domain with a center $0$ (see \cite[Proposition 3.8]{I-K}). 
Therefore, every bounded homogeneous domain is biholomorphic to a minimal bounded homogeneous domain. 

\subsection{Berezin symbol}
We fix a minimal bounded homogeneous domain $\mathcal{U}$ with a center $\cent$. 
For a bounded linear operator $T$ on $L^2_a(\mathcal{U})$, the Berezin symbol $\widetilde{T}$ of $T$ is defined by 
\begin{eqnarray*}
 \widetilde{T}(z) := \inner{T k_z}{k_z}  \ \ (z \in \mathcal{U})  . 
\end{eqnarray*}
For a Borel measure $\mu$ on $\mathcal{U}$, we define a function $\widetilde{\mu}$ on $\mathcal{U}$ by
\begin{eqnarray*}
 \widetilde{\mu}(z) := \int_{\mathcal{U}} \vert k_z(w) \vert^2 \, d\mu(w) ,
\end{eqnarray*}
which is called the Berezin symbol of the measure $\mu$. 
Since $\n{K_{\mathcal{U}}(z,w)}$ is a bounded function on $B(\cent,\ru) \times \mathcal{U}$ (see \cite[Proposition 6.1]{I-Y}), 
$\widetilde{\mu}$ is a continuous function if $\mu$ is finite. 

Suppose that the Toeplitz operator $T_{\mu}$ is a bounded operator on $L^2_a(\mathcal{U})$. We have 
\begin{eqnarray*}
 \widetilde{T_{\mu}}(z) =   \inner{T_{\mu}k_z}{k_z} 
                        = \frac{1}{K_{\mathcal{U}}(z,z)^{1/2}} T_{\mu}k_z (z) 
\end{eqnarray*}
by the definition of the reproducing kernel. The right hand side equals
\begin{eqnarray*}
  \frac{1}{K(z,z)^{1/2}} \int_{\mathcal{U}} K_{\mathcal{U}}(z,w) k_z (w) \, d\mu(w) 
                    = \int_{\mathcal{U}} \vert k_z(w) \vert^2 \, d\mu(w) .
\end{eqnarray*}
Therefore, we have 
\begin{eqnarray}
\widetilde{T_{\mu}}(z) =  \widetilde{\mu}(z) .   \label{tilde}
\end{eqnarray}

\subsection{Carleson measure and vanishing Carleson measure}
Let $\mu$ be a positive Borel measure on $\mathcal{U}$ and $p \geq 1$. 
We say that $\mu$ is a Carleson measure for $L^p_a(\mathcal{U})$ if 
there exists a constant $M>0$ such that 
\begin{eqnarray*}
 \int_{\mathcal{U}} \vert f(z) \vert^p \, d\mu (z) \leq M \int_{\mathcal{U}} \vert f(z) \vert^p \, dV(z) 
\end{eqnarray*}
for all $f \in L^p_a(\mathcal{U})$. It is easy to see that $\mu$ is a Carleson measure for $L^p_a(\mathcal{U})$ if and only if $L^p_a(\mathcal{U}) \subset L^p_a(\mathcal{U}, d\mu)$ and the inclusion map 
$$  i_p : L^p_a(\mathcal{U}) \longrightarrow  L^p_a(\mathcal{U}, d\mu)$$
is bounded. 

Suppose $\mu$ is a Carleson measure for $L^2_a(\mathcal{U})$. 
We say that $\mu$ is a vanishing Carleson measure for $L^2_a(\mathcal{U})$ if the inclusion map 
$$  i_2 : L^2_a(\mathcal{U}) \longrightarrow  L^2_a(\mathcal{U}, d\mu)$$
is compact.

\subsection{Boundedness of the positive Bergman operator}
In order to prove the part $(c) \Longrightarrow (a)$ in Theorem $\ref{Bddness*}$, we use the boundedness of the positive Bergman operator $P_{\mathcal{U}}^+$ on $L^2(\mathcal{U},dV)$ defined by 
\begin{eqnarray}
P_{\mathcal{U}}^+ g(z) := \int_{\mathcal{U}} \n{K_{\mathcal{U}}(z,w)} g(w) \, dV(w)  \label{positive Bergman operator}
\end{eqnarray}
for $g \in L^2(\mathcal{U},dV)$. 
We prove that $P_{\mathcal{U}}^+$ is a bounded operator on $L^2(\mathcal{U},dV)$.

It is known that every bounded homogeneous domain is holomorphically equivalent to a homogeneous Siegel domain \cite{VGS}. 
Let $\Phi$ be a biholomorphic map from $\mathcal{U}$ to a Siegel domain $\mathcal{D}$. We define a unitary map $U_{\Phi}$ 
from $L^2(\mathcal{U},dV)$ to $L^2(\mathcal{D},dV)$ by 
$$ U_{\Phi} f(\zeta)  := f(\Phi^{-1}(\zeta)) \n{\det J(\Phi^{-1},\zeta)} . $$
Then, we have
$$    U_{\Phi} \circ P_{\mathcal{U}}^+ = P_{\mathcal{D}}^+ \circ U_{\Phi}  \ \ \ (f \in L^2(\mathcal{U},dV)). $$
Therefore, the boundedness of $P_{\mathcal{U}}^+$ on $L^2(\mathcal{U},dV)$ is equivalent to the boundedness of $P_{\mathcal{D}}^+$ on $L^2(\mathcal{D},dV)$. 
On the other hand, B{\'e}koll{\'e}-Kagou proved the boundedness of the positive Bergman operator $P_{\mathcal{D}}^+$ on $L^2(\mathcal{D},dV)$
 (\cite[Theorem II.7]{B-K}). 
Therefore, we have the following lemma. 

\begin{lem} \label{PosBergBddness} 
The operator $P_{\mathcal{U}}^+$ is bounded on $L^2(\mathcal{U},dV)$. 
\end{lem}

\section{Some Lemmas}
In this section, we show some lemmas for a minimal bounded homogeneous domain $\mathcal{U}$ with a center $\cent \in \mathcal{U}$. 
Although the proofs of these lemmas are almost same as the ones for the case of symmetric domain (\cite{BCZ},\cite{ZhuBMO},\cite{Zhu3}), 
we write them here for the sake of completeness. 
In this section, $K(z,w)$ means $K_{\mathcal{U}}(z,w)$. 
First, we present the following theorem, which plays fundamental roles in this work. 
%%%%%%%%%%%%%%%%%%%%%%%%%%%%%%%%%%%%%%%%%%%%%%%%%%%%%%%%%%%%%%%%%%%%%%%%%%%%%
\begin{thm}[\mbox{\cite[Theorem A]{I-Y}}] \label{Ishi} 
For any $\ru>0$, there exists $C_{\ru}>0$ such that 
$$  C_{\ru}^{-1} \leq \n{\frac{K(z,a)}{K(a,a)} }\leq C_{\ru} $$
for all $z,a \in \mathcal{U}$ such that $\beta (z,a) \leq \ru$.\\
\end{thm}
%%%%%%%%%%%%%%%%%%%%%%%%%%%%%%%%%%%%%%%%%%%%%%%%%%%%%%%%%%%%%%%%%%%%%%%%%%%%%

For $a \in \mathcal{U}$, let $\varphi_a$ be an automorphism of $\mathcal{U}$ such that $\varphi_a (a) =\cent$. 
Using Theorem \ref{Ishi}, we prove Theorem $\ref{thmCarleson}$. First, we prove some lemmas. 
\begin{lem} \label{HomBergKer} 
One has 
\begin{eqnarray}
\n{\det J (\varphi_a, z)}^2       &=&  \frac{\n{K(z,a)}^2}{K(\cent,\cent) K(a,a)} ,  \label{Jacob-01} \\
\n{\det J (\varphi_a^{-1}, z)}^2  &=&  \frac{K(\cent,\cent) K(a,a)}{\n{K(\varphi_a^{-1}(z),a)}^2},  \label{Jacob-02}
\end{eqnarray}
where $\det J (\varphi_a, z)$ is the complex Jacobian of $\varphi_a$ at $z$.
\end{lem}
\begin{proof}
By the transformation formula of the Bergman kernel, we have 
\begin{eqnarray*}
K(z,a) = K(\varphi_a(z),\varphi_a(a)) \, \det J (\varphi_a, z) \overline{\det J (\varphi_a, a)} .
\end{eqnarray*}
Since $K(\varphi_a(z),\varphi_a(a)) = K(\varphi_a(z),\cent) = K(\cent,\cent)$, we obtain 
\begin{eqnarray}
\n{\det J (\varphi_a, z)}^2 =  \frac{\n{K(z,a)}^2}{K(\cent,\cent)^2 \, \n{\det J (\varphi_a, a)}^2} .  \label{Jacob-1} 
\end{eqnarray}
On the other hand, we have
\begin{eqnarray*}
K(a,a) = K(\varphi_a(a),\varphi_a(a)) \, \n{\det J (\varphi_a, a)}^2  .  
\end{eqnarray*}
This means 
\begin{eqnarray}
 \n{\det J (\varphi_a, a)}^2 = \frac{K(a,a)}{K(\cent,\cent)}  .  \label{Jacob-2} 
\end{eqnarray}
From $(\ref{Jacob-1})$ and $(\ref{Jacob-2})$, we obtain $(\ref{Jacob-01})$. 
The equality $(\ref{Jacob-02})$ follows from 
$$ \det J (\varphi_a, \varphi_a^{-1} (z)) \, \det J ( \varphi_a^{-1} , z) =1 .  \eqno\qedhere $$
\end{proof}

For any $z \in \mathcal{U}$ and $\ru>0$, let
$$  B(z,\ru) := \{ w \in \mathcal{U} \mid \beta(z,w) \leq \ru \} $$
be the Bergman metric disk with center $z$ and radius $\ru$. 

\begin{lem}[cf. \mbox{\cite[Lemma 8]{BCZ}}]  \label{Lemma1} 
There exists a constant $\Kr$ such that 
$$  \Kr^{-1} \leq \n{k_a (z)}^2 \Vol{B(a,\ru)} \leq \Kr $$
for all $a \in \mathcal{U}$ and $z \in B(a,\ru)$.
\end{lem}
\begin{proof}
Thanks to the invariance of the Bergman distance under biholomorphic transformations, we have 
$$ \Vol{B(a,\ru)} = \int_{B(\cent,\ru)} \n{\det J (\varphi_a^{-1} ,u)}^2 \, dV(u) . $$
By Lemma $\ref{HomBergKer}$, we obtain 
\begin{eqnarray}
 \n{k_a (z)}^2 \Vol{B(a,\ru)} 
       &=&   \frac{\n{K(z,a)}^2}{K(a,a)} \int_{B(\cent,\ru)} \frac{K(\cent,\cent) K(a,a)}{\n{K(\varphi_a^{-1}(u),a)}^2} \, dV(u) \nonumber \\
       &=& K(\cent,\cent) \int_{B(\cent,\ru)} \frac{\n{K(z,a)}^2}{\n{K(\varphi_a^{-1}(u),a)}^{2}} \, dV(u)    .   \label{Lemma1-1}
\end{eqnarray}
Since $u \in B(\cent,\ru)$ means $\beta(\cent,u) \leq \ru$, 
we have $\beta(a,\varphi_a^{-1}(u)) \leq \ru$, so that Theorem $\ref{Ishi}$ implies 
\begin{eqnarray}
C_{\ru}^{-1} \leq  \n{ \frac{K(a,a)}{K(\varphi_a^{-1}(u),a)}} \leq C_{\ru} . \label{Lemma1-2}
\end{eqnarray}
On the other hand, we have
\begin{eqnarray}
C_{\ru}^{-1} \leq \n{\frac{K(z,a)}{K(a,a)}} \leq C_{\ru} .  \label{Lemma1-3}
\end{eqnarray}
Multiplying  $(\ref{Lemma1-2})$ by $(\ref{Lemma1-3})$, we obtain 
\begin{eqnarray}
C_{\ru}^{-2} \leq  \frac{\n{K(z,a)}}{\n{K(\varphi_a^{-1}(u),a)}} \leq C_{\ru}^2 . \label{Lemma1-4}
\end{eqnarray}
By $(\ref{Lemma1-1})$ and $(\ref{Lemma1-4})$, we complete the proof with $\Kr = C_{\ru}^2 K(\cent,\cent) {\rm Vol \, }(B(\cent,\ru))$. 
\end{proof}

Since one uses not the symmetry but the homogeneity of a complex domain 
in the proof of \cite[Lemma 5]{ZhuBMO}, the following lemma holds for the minimal bounded homogeneous domain $\mathcal{U}$. %general 

\begin{lem}[\mbox{\cite[Lemma 5]{ZhuBMO}}]  \label{Lemma4}
There exists a sequence $\{ w_j \} \subset \mathcal{U}$ satisfying the following conditions. \\
$(S1) \ \mathcal{U} = \cup^{\infty}_{j=1} B(w_j,\ru). $\\
$(S2) \ B(w_i,\ru/4) \cap B(w_j,\ru/4) = \emptyset.$\\
$(S3)$ There exists a positive integer $N$ such that each point $z \in \mathcal{U}$ belongs to at most $N$ of the sets $B(w_j,2\ru)$.
\end{lem}

\begin{lem}[cf. \mbox{\cite[Lemma 7]{ZhuBMO}} ] \label{Lemma5}
There exists a constant $C$ such that 
\begin{eqnarray}
  \n{f(a)}^p \leq \frac{C}{\Vol{B(a,\ru)}} \int_{B(a,\ru)} \n{f(z)}^p \, dV(z)  \label{Lemma5.1}
\end{eqnarray}
for all $f \in \mathcal{O} (\mathcal{U}), \, p \geq 1$ and $a \in \mathcal{U}$.
\end{lem}

\begin{proof}
First, we consider the case $a=\cent$. 
Since the Bergman metric induces the usual Euclidean topology on $\mathcal{U}$, 
there exists a Euclidean ball $E(\cent,R)$ with center $\cent$ and the radius $R$ such that $E(\cent,R) \subset B(\cent,\ru)$. 
Let $f$ be a holomorphic function on $\mathcal{U}$. 
Since $f$ has a mean value property, we have 
\begin{eqnarray*}
f(\cent) = \frac{1}{\Vol{E(\cent,R)}} \int_{E(\cent,R)} f(z) \, dV(z) .
\end{eqnarray*}
Therefore, we have 
\begin{eqnarray}
\n{f(\cent)}^p &\leq& \left( \frac{1}{\Vol{E(\cent,R)}} \int_{E(\cent,R)} \n{f(z)} \, dV(z) \right)^p  \nonumber \\
           &\leq& \left( \frac{1}{\Vol{E(\cent,R)}} \right)^p \left( \N{f}_{L^p(E(\cent,R))} \N{1}_{L^q(E(\cent,R))} \right)^p , \label{Meanvalue0.5}
\end{eqnarray}
where $q$ denotes the conjugate exponent of $p$. 
Since 
$$\N{1}_{L^q(E(\cent,R))}^p = \Vol{E(\cent,R)}^{\frac{p}{q}}, $$ 
the last term of $(\ref{Meanvalue0.5})$ is equal to 
\begin{eqnarray*}
      \left( \Vol{E(\cent,R)} \right)^{-p+\frac{p}{q}}  \int_{E(\cent,R)} \n{f(z)}^p \, dV(z) .
\end{eqnarray*}
Therefore, we have  
\begin{eqnarray*}
\n{f(\cent)}^p \leq \frac{1}{\Vol{E(\cent,R)}}   \int_{E(\cent,R)} \n{f(z)}^p \, dV(z) 
\end{eqnarray*}
because ${-p+\frac{p}{q}} = p(-1+\frac{1}{q}) = 1$. 

Now, put $C_R := \frac{1}{\Vol{E(\cent,R)}}$. Note that the constant $C_R$ is independent of $p$ and $f$. 
Since $E(\cent,R) \subset B(\cent,\ru)$, we have 
\begin{eqnarray}
\n{f(\cent)}^p \leq C_R  \int_{B(\cent,\ru)} \n{f(z)}^p \, dV(z) .  \label{Meanvalue1}
\end{eqnarray}
Next, we prove the general case. 
Since $f \circ \varphi_a^{-1}$ is a holomorphic function on $\mathcal{U}$, we have 
\begin{eqnarray}
\n{f \circ \varphi_a^{-1} (\cent)}^p \leq C_R  \int_{B(\cent,\ru)} \n{f \circ \varphi_a^{-1}(z)}^p \, dV(z)   \label{Meanvalue2}
\end{eqnarray}
by $(\ref{Meanvalue1})$. Put $w := \varphi_a^{-1}(z)$. Then the inequality $(\ref{Meanvalue2})$ means 
\begin{eqnarray*}
\n{f(a)}^p \leq C_R  \int_{B(a,\ru)} \n{f(w)}^p \, \n{\det J (\varphi_a, w)}^2 \, dV(w) . 
\end{eqnarray*}
By Lemma $\ref{HomBergKer}$, the right hand side is equal to 
\begin{eqnarray*}
C_R  \int_{B(a,\ru)} \n{f(w)}^p \frac{\n{K(w,a)}^2}{K(\cent,\cent) K(a,a)} \, dV(w) . 
\end{eqnarray*}
Therefore we have 
\begin{eqnarray}
\n{f(a)}^p \leq C_R \frac{K(a,a)}{K(\cent,\cent)} \int_{B(a,\ru)} \n{f(w)}^p \n{\frac{K(w,a)}{ K(a,a)}}^2  \, dV(w) .  \label{Meanvalue4}
\end{eqnarray}
By Theorem $\ref{Ishi}$, we have 
\begin{eqnarray}
C_{\ru}^{-2} \leq \n{\frac{K(w,a)}{ K(a,a)}}^2 \leq C_{\ru}^2  \label{Meanvalue5}
\end{eqnarray}
on $w \in B(a,\ru)$. 
Therefore we have 
\begin{eqnarray}
\n{f(a)}^p \leq C_R C_{\ru}^2 \frac{K(a,a)}{K(\cent,\cent)} \int_{B(a,\ru)} \n{f(w)}^p  \, dV(w)    \label{Meanvalue6}
\end{eqnarray}
by $(\ref{Meanvalue4})$ and $(\ref{Meanvalue5})$. 
We see from $(\ref{Meanvalue5})$ and Lemma \ref{Lemma1} that 
\begin{eqnarray*}
C_{\ru}^{-2} \leq \n{\frac{K(w,a)}{ K(a,a)}}^2 
= \frac{\n{k_a(w)}^2}{ K(a,a)} \leq \frac{\Kr}{\Vol{B(a,\ru)} \, K(a,a)} .
\end{eqnarray*}
Hence we obtain 
\begin{eqnarray}
 K(a,a) \leq   \frac{\Kr C_{\ru}^{2}}{\Vol{B(a,r)} } .   \label{Meanvalue7}
\end{eqnarray}
By $(\ref{Meanvalue6})$ and $(\ref{Meanvalue7})$, we have 
\begin{eqnarray*}
\n{f(a)}^p \leq   \frac{C}{\Vol{B(a,\ru)} } \int_{B(a,\ru)} \n{f(w)}^p  \, dV(w) .  
\end{eqnarray*}
with $C = C_{\ru}^4 C_R \Kr K(\cent,\cent)^{-1}$.
\end{proof}

%%%%%%%%%%%%%%%%%%%%%%%%%%%%%%%%%%%%%%%%%%%%%%%%%%%%%%%%%%%%%%%%%%
\begin{lem} \label{PropCarl}
There exists a constant $C$ such that 
\begin{eqnarray*}
\sup_{w \in B(a,\ru)} \n{f(w)}^p 
  \leq \frac{C}{\Vol{B(a,\ru)}}  \int_{B(a,2\ru)} \n{f(z)}^p \, dV(z) 
\end{eqnarray*}
for all $f \in \mathcal{O} (\mathcal{U}), p \geq 1$ and $a \in \mathcal{U}$.
\end{lem}

\begin{proof}
By Lemma \ref{Lemma5}, there exists a constant $C$ such that 
$$  \n{f(w)}^p \leq \frac{C}{\Vol{B(w,\ru)}} \int_{B(w,\ru)} \n{f(z)}^p \, dV(z)  $$
for any $f \in \mathcal{O} (\mathcal{U}), \, p \geq 1$ and $w \in \mathcal{U}$. 
Therefore we have
\begin{eqnarray*}
\sup_{w \in B(a,\ru)} \n{f(w)}^p 
  &\leq& C \, \sup_{w \in B(a,\ru)}  \left( \frac{1}{\Vol{B(w,\ru)}} \int_{B(w,\ru)} \n{f(z)}^p \, dV(z) \right) \nonumber \\
  &\leq& C \, \left( \int_{B(a,2\ru)} \n{f(z)}^p \, dV(z) \right) \sup_{w \in B(a,\ru)} \frac{1}{\Vol{B(w,\ru)}}  , 
\end{eqnarray*}
where the last inequality holds because $B(w,\ru)$ is a subset of $B(a,2\ru)$ for all $ w \in B(a,\ru)$.  
Hence, it is sufficient to prove
\begin{eqnarray*}
\sup_{w \in B(a,\ru)} \frac{1}{\Vol{B(w,\ru)}} \leq \frac{C}{\Vol{B(a,\ru)}} . 
\end{eqnarray*}

Take any $w \in B(a,\ru)$ and let $b \in B(a,\ru) \cap B(w,\ru)$. Then we have 
\begin{eqnarray*}
  \Vol{B(a,\ru)} &\leq& \Kr  \n{k_{a}(b)}^{-2} , \\
  \Vol{B(w,\ru)} &\geq& \Kr^{-1} \n{k_w(b)}^{-2}  
\end{eqnarray*}
by Lemma $\ref{Lemma1}$. Therefore, we obtain
\begin{eqnarray}
  \frac{\Vol{B(a,\ru)}}{\Vol{B(w,\ru)}} \leq \Kr^2 \n{\frac{k_w(b)}{k_{a}(b)}}^2 . \label{PropCarl1}
\end{eqnarray}
On the other hand, we have 
\begin{eqnarray*}
 \n{\frac{k_w(b)}{k_{a}(b)}}^2 
 &=&  \frac{\n{K(w,b)}^2}{K(w,w)} \frac{K(a,a)}{\n{K(a,b)}^2} \\
 &=&  \n{\frac{K(w,a)}{K(w,w)}} \n{\frac{K(a,a)}{K(w,a)}} \n{\frac{K(w,b)}{K(b,b)}}^2 \n{\frac{K(b,b)}{K(a,b)}}^2   . 
\end{eqnarray*}
Since $\beta(w,a), \, \beta(w,b)$ and $\beta(a,b)$ do not exceed $\ru$, we have 
\begin{eqnarray}
 \n{\frac{k_w(b)}{k_{a}(b)}}^2 \leq C_{\ru}^6 \label{PropCarl3}
\end{eqnarray}
by Theorem $\ref{Ishi}$. Therefore, we have 
\begin{eqnarray}
\sup_{w \in B(a,\ru)} \frac{1}{\Vol{B(w,\ru)}} \leq \frac{C}{\Vol{B(a,\ru)}} \label{PropCarl4}
\end{eqnarray}
by $(\ref{PropCarl1})$ and $(\ref{PropCarl3})$.
\end{proof}

By Lemmas \ref{Lemma1}, \ref{Lemma4} and \ref{PropCarl}, we can prove the following theorem 
as in the same way of the proof of \cite[Theorem 7]{Zhu1}. 
It follows from this theorem that the property of being a Carleson measure is independent of $p$. 
\begin{thm}[\mbox{\cite[Theorem 7]{Zhu1}}] \label{thmCarleson}  
Suppose $\mu$ is a positive Borel measure on $\mathcal{U}$ and $p \geq 1$. 
Then $\mu$ is a Carleson measure for $L^p_a(\mathcal{U})$ if and only if 
\begin{eqnarray}
\sup_{a \in \mathcal{U}} \frac{\mu(B(a,\ru))}{\Vol{B(a,\ru)}} < \infty. \label{Carlesonequiv0}
\end{eqnarray}
\end{thm}

It is known that $\mathcal{H} := {\rm span} \langle K_{\mathcal{U}}(\cdot,w) \rangle_{w \in \mathcal{U}}$ is dense in $L^2_a(\mathcal{U})$. 
On the other hand, $K_{\mathcal{U}}(\cdot,w)$ is bounded for each $w \in \mathcal{U}$ (see \cite[Proposition 6.1]{I-Y}). 
Therefore $\mathcal{H}  \subset H^{\infty}$, so that  $H^{\infty}$ is dense in $L^2_a(\mathcal{U})$. 
Since $K(a,a) \rightarrow \infty$ as $a \rightarrow \partial \mathcal{U}$ (see \cite[Proposition 5.2]{Koba}), 
we can prove the following lemmas in the same way as in \cite{Eng}. 

\begin{lem}[\mbox{\cite[Lemma 1]{Eng}}] \label{weak0}  
A sequence $\{ k_a \}$ converges to $0$ weakly in $L^2_a(\mathcal{U})$ as $a \rightarrow \partial \mathcal{U}$.
\end{lem}
\begin{lem}[\mbox{\cite[Lemma 5]{Eng}}] \label{kougiitiyou} 
Let $\{ f_n \}$ be a sequence of functions in $L^2_a(\mathcal{U})$ which is weakly convergent to $f$.
Then $f_n \rightarrow f$ uniformly on compact subsets of $\mathcal{U}$. 
\end{lem}

From Lemma $\ref{weak0}$ and $\ref{kougiitiyou}$, we can prove the following theorem.

\begin{thm}[\mbox{\cite[Theorem 11]{Zhu1}, \cite[Theorem 7.7]{Zhu3}}]   \label{claim:7.2} 
Let $\mu$ be a finite positive Borel measure on $\mathcal{U}$. 
Then $\mu$ is a vanishing Carleson measure for $L^2_a(\mathcal{U})$ if and only if 
\begin{eqnarray*}
\lim_{a \rightarrow \partial \mathcal{U}} \frac{\mu(B(a,\ru))}{\Vol{B(a,\ru)}} =0. 
\end{eqnarray*}
\end{thm}

%%%%%%%%%%%%%%%%%%%%%%%%%%%%%%%%%%%%%%%%%%%%%%%%%%%%%%%%%%%%%%%%%%%%%%%%%%
\section{Boundedness of the Toeplitz operator}
%%%%%%%%%%%%%%%%%%%%%%%%%%%%%%%%%%%%%%%%%%%%%%%%%%%%%%%%%%%%%%%%%%%%%%%%%%
In this section, we prove the main theorem. 
\begin{thm} \label{MainBddness}
Let $\mathcal{U} \subset \C^n$ be a minimal bounded homogeneous domain 
and $\mu$ a positive Borel measure on $\mathcal{U}$. 
Then the following conditions are all equivalent.\\
$(a)$ \ $T_{\mu}$ is a bounded operator on $L^2_a(\mathcal{U})$.\\
$(b)$ \ $\widetilde{\mu}(z) $ is a bounded function on $\mathcal{U}$.\\
$(c)$ \ For all $p \geq 1$, $\mu$ is a Carleson measure for $L^p_a(\mathcal{U})$. \\
$(d)$ \ $\widehat{\mu}(z) $ is a bounded function on $\mathcal{U}$.
\end{thm}
\begin{proof}
We have already proved $(c) \Longleftrightarrow (d)$ in Theorem $\ref{thmCarleson}$. 
We will prove $(a) \Longrightarrow (b) \Longrightarrow (d)$ and $(c) \Longrightarrow (a)$. 

First, we prove $(a) \Longrightarrow (b)$. 
Since $T_{\mu}$ is a bounded operator, we have 
\begin{eqnarray*}
\widetilde{\mu}(z) = \widetilde{T_{\mu}}(z)  = \n{ \inner{T_{\mu}k_z}{k_z}} \leq \N{T_{\mu}} \N{k_z}^2 = \N{T_{\mu}} < \infty ,
\end{eqnarray*}
where the first equality follows from $(\ref{tilde})$. 

Next, we prove $(b) \Longrightarrow (d)$. 
By Lemma $\ref{Lemma1}$, we have
\begin{eqnarray*}
\Kr^{-1} \leq \n{k_z (w)}^2 \Vol{B(z,\ru)}.
\end{eqnarray*}
We integrate this inequality on $B(z,\ru)$ by $\mu$. 
Then we have
\begin{eqnarray*}
\Kr^{-1} \int_{B(z,\ru)} \, d\mu(w) \leq \Vol{B(z,\ru)} \int_{B(z,\ru)} \n{k_z (w)}^2 \, d\mu(w) .
\end{eqnarray*}
Therefore, we have 
\begin{eqnarray*}
\frac{\mu(B(z,\ru))}{ \Vol{B(z,\ru)}}  &\leq& \Kr \int_{B(z,\ru)} \vert k_z(w) \vert^2 \, d\mu(w)  \\
  &\leq& \Kr \N{k_z}^2_{L^2(d\mu)} = \Kr \,  \widetilde{\mu}(z) .
\end{eqnarray*}
Therefore we have $\widehat{\mu}(z) \leq \Kr \,  \widetilde{\mu}(z)$, so $\widehat{\mu}(z) $ is a bounded function on $\mathcal{U}$.

Finally, we prove $(c) \Longrightarrow (a)$. 
For $f \in L^2_a(\mathcal{U})$, we have 
\begin{eqnarray}
\N{T_{\mu}f}_2^2 
               &=& \int_{\mathcal{U}} \n{ \int_\mathcal{U} K_{\mathcal{U}}(z,w) f(w) \, d\mu(w)}^2 dV(z)  \nonumber \\
            &\leq& \int_{\mathcal{U}} \left( \int_\mathcal{U} \n{K_{\mathcal{U}}(z,w)} \n{ f(w)} \, d\mu(w) \right)^2 dV(z)   \nonumber \\
               &=& \int_{\mathcal{U}} \left( \int_\mathcal{U} \n{F_z(w)} \, d\mu(w) \right)^2 dV(z),   \label{Mainthm1siki}
\end{eqnarray}
where we put $F_z(w) := \overline{K_{\mathcal{U}}(z,w)}f(w)$. 
Since $\overline{K_{\mathcal{U}}(z,\cdot)} \in L^2_a({\mathcal{U}})$, we have $F_z \in L_a^1(\mathcal{U})$. 
Moreover, $\mu$ is a Carleson measure. 
Hence, there exists a positive constant $M_{\mu}$ such that 
\begin{eqnarray}
 \int_{\mathcal{U}} \n{F_z(w)} \, d\mu (w) \leq M_{\mu} \int_{\mathcal{U}} \n{F_z(w)} \, dV(w). \label{Mainthm1.5siki}
\end{eqnarray}
By the definition of the Carleson measure, $M_{\mu}$ is independent of $z$.
Therefore, we have 
\begin{eqnarray}
\N{T_{\mu}f}^2_2
   \leq  M_{\mu}^2 \, \int_{\mathcal{U}} \left( \int_\mathcal{U} \n{K_{\mathcal{U}}(z,w)} \n{ f(w)} \, dV(w) \right)^2 dV(z) \label{Mainthm2siki}
\end{eqnarray}
by $(\ref{Mainthm1siki})$ and $(\ref{Mainthm1.5siki})$. 
Moreover, the right hand side is rewritten as $M_{\mu}^2 \N{P_{\mathcal{U}}^+ f^+}_2^2$, where $f^+=\n{f}$. 
Since $P_{\mathcal{U}}^+$ is a bounded operator by Theorem $\ref{PosBergBddness}$, we have 
$$\N{T_{\mu}f}_2 \leq M_{\mu} \N{P_{\mathcal{U}}^+ f^+}_2 \leq M_{\mu} \N{P_{\mathcal{U}}^+}\N{f}_2.$$ 

Next, we prove $T_{\mu}f \in \mathcal{O}(\mathcal{U})$. 
Since $T_{\mu}f \in L^2(\mathcal{U})$, it is enough to prove $\inner{T_{\mu} f}{g} =0$ for any $g \in L^2_a(\mathcal{U})^{\perp}$. 
We see that 
\begin{eqnarray}
\inner{T_{\mu} f}{g} 
   &=& \int_{\mathcal{U}} \left\{ \int_{\mathcal{U}} K_{\mathcal{U}}(z,w) f(w) \, d\mu(w) \right\} \overline{g(z)} \, dV(z) \nonumber \\
   &=& \int_{\mathcal{U}} \overline{\left\{ \int_{\mathcal{U}} K_{\mathcal{U}}(w,z) g(z)\, dV(z) \right\} } f(w) \, d\mu(w)  \nonumber \\
   &=& 0 .  \label{FubiniSiyou}
\end{eqnarray}
Note that since 
\begin{eqnarray}
 \int_{\mathcal{U}}  \int_{\mathcal{U}} \n{K_{\mathcal{U}}(w,z) g(z) f(w)}\, d\mu(w) dV(z)    
   \leq M_{\mu} \N{P_{\mathcal{U}}^+} \N{f}_2 \N{g}_2 < \infty ,  \label{Fubini0}
\end{eqnarray}
the second equality of $(\ref{FubiniSiyou})$ follows from Fubini's theorem. 

Therefore, $T_{\mu}$ is a bounded operator on $ L^2_a({\mathcal{U}})$.
\end{proof}

%%%%%%%%%%%%%%%%%%%%%%%%%%%%%%%%%%%%%%%%%%%%%%%%%%%%%%%%%%%%%%%%%%%%%%%%%%%%%%%%%%%%%%%%%%%%%%%%%%%%%%%%%%%%%%%%%%%
\section{Compactness of the Toeplitz operator}
%%%%%%%%%%%%%%%%%%%%%%%%%%%%%%%%%%%%%%%%%%%%%%%%%%%%%%%%%%%%%%%%%%%%%%%%%%%%%%%%%%%%%%%%%%%%%%%%%%%%%%%%%%%%%%%%%%%
Suppose $1<p<\infty$ and $q$ is the conjugate exponent of $p$. 
It is known that $(L^p_a(\mathbb{D}))^{\ast} \cong L^q_a(\mathbb{D})$ with equivalent norms and under the integral pairing: 
\begin{align}
\inner{f}{g} = \int_{\mathbb{D}} f(z) \overline{g(z)} \, dV(z) ,
\end{align}
where $f \in L^p_a(\mathbb{D})$ and $g \in L^q_a(\mathbb{D})$ (see \cite[Theorem 4.25]{Zhu3}). 
To prove this, we use the boundedness of the positive Bergman projection $P^+_{\mathbb{D}}$ on $L^p(\mathbb{D},dV)$. 
But, we do not know that $P^+_{\mathcal{U}}$ is a bounded operator on $L^p(\mathcal{U},dV)$ for $p \neq 2$, 
whereas the similar statement is shown for homogeneous Siegel domain by B{\'e}koll{\'e}-Kagou. 
Therefore, we consider the case $p=2$ in the present work. 

\begin{thm} \label{MainCptness}  
Let $\mathcal{U}$ be a minimal bounded homogeneous domain and $\mu$ a finite positive Borel measure on $\mathcal{U}$. 
Then the following conditions are all equivalent.\\
$(a)$ \ $T_{\mu}$ is a compact operator on $L^2_a(\mathcal{U})$.\\
$(b)$ \ $\widetilde{\mu}(z)  \rightarrow 0$ as $z \rightarrow \partial \mathcal{U}$.\\ 
$(c)$ \ $\mu$ is a vanishing Carleson measure for $L^2_a(\mathcal{U})$. \\
$(d)$ \ $\widehat{\mu}(z)  \rightarrow 0$ as $z \rightarrow \partial \mathcal{U}$. 
\end{thm}

\begin{proof} 
Theorem $\ref{claim:7.2}$ shows $(c) \Longleftrightarrow (d)$. 
We will prove $(a) \Longrightarrow (b) \Longrightarrow (d)$ and $(c) \Longrightarrow (a)$. 

First, we prove that $(a) \Longrightarrow (b)$.
By Lemma $\ref{weak0}$, we have $k_z \rightarrow  0 $ weakly in $L^2_a(\mathcal{U})$ as $z \rightarrow \partial \mathcal{U}$. 
Since $T_{\mu}$ is a compact operator, we have $T_{\mu} k_z \rightarrow 0 $ in $L^2_a(\mathcal{U})$. 
Therefore, we have 
$$  \widetilde{\mu}(z) =  \n{\inner{T_{\mu}k_z}{k_z} } \leq \Vert T_{\mu} k_z \Vert_2 \Vert k_z \Vert_2
        = \N{ T_{\mu} k_z }_2 \longrightarrow 0 \ (z \rightarrow \partial \mathcal{U}) . $$ 
        
Next, we prove $(b) \Longrightarrow (d)$. 
We have already shown that 
\begin{eqnarray}
\widehat{\mu}(z) \leq  \Kr \, \widetilde{\mu}(z) 
\end{eqnarray}
in the proof of Theorem $\ref{MainBddness}$. Therefore, we have $\widehat{\mu}(z)  \rightarrow 0$ as $z \rightarrow \partial \mathcal{U}$. 

Finally, we prove $(c) \Longrightarrow (a)$. 
First, we prove that $ \N{T_{\mu} f}_{L^2(dV)} \leq M_{\mu} \N{f}_{L^2(d\mu)} $ for any $f \in L^2_a(\mathcal{U})$. 
Since $\mu$ is a Carleson measure, we have $T_{\mu} f\in L^2_a(\mathcal{U})$ by Theorem $\ref{MainBddness}$.

Take any $g \in L^2_a(\mathcal{U})$. Then, we have
\begin{eqnarray*}
\inner{T_{\mu}f}{g} &=& \int_{\mathcal{U}} \left( \int_{\mathcal{U}} K_{\mathcal{U}}(z,w) f(w) \, d\mu(w) \right) \overline{g(z)} \, dV(z) \\
                        &=& \int_{\mathcal{U}} \left( \int_{\mathcal{U}} K_{\mathcal{U}}(z,w) \overline{g(z)} \, dV(z) \right) f(w) \, d\mu(w) \\
                        &=& \int_{\mathcal{U}}  f(w) \overline{g(w)}  \, d\mu(w) .
\end{eqnarray*}
Note that we can change the order of integral because $(\ref{Fubini0})$ holds for the case $g \in L^2_a(\mathcal{U})$. 
Since 
\begin{eqnarray*}                        
 \n{\inner{T_{\mu}f}{g} } \leq \N{f}_{L^2(d\mu)} \N{g}_{L^2(d\mu)} \leq  M_{\mu} \N{f}_{L^2(d\mu)} \N{g}_{L^2(dV)},
\end{eqnarray*}
we have 
\begin{eqnarray}
 \Vert T_{\mu} f \Vert_{2} \leq M_{\mu} \Vert f \Vert_{L^2(d\mu)} .   \label{CptnessNorm}
\end{eqnarray}

Next, we prove the compactness of $T_{\mu}$. 
Take any sequence $\{ f_n \}$ such that $f_n \rightarrow 0$ weakly in $L^2_a(\mathcal{U})$. 
Since $\mu$ is a vanishing Carleson measure for $L^2_a(\mathcal{U})$, we have $f_n \rightarrow 0$ in $L^2_a(\mathcal{U},d\mu)$. 
Therefore we have $\N{T_{\mu}f_n}_2 \rightarrow 0$ by $(\ref{CptnessNorm})$. It means that $T_{\mu}$ is a compact operator on $L^2_a(\mathcal{U})$.
\end{proof}

% ------------------------------------------------------------------------

\subsection*{Acknowledgment}
The author would like to thank to Professor H.~Ishi for useful advices 
and to Professors T.~Ohsawa and N.~Suzuki for helpful discussions.

\ \\
Satoshi Yamaji\\
Graduate School of Mathematics, Nagoya University, \\
Chikusa-ku, Nagoya, 464-8602, Japan\\
E-mail address: {satoshi.yamaji@math.nagoya-u.ac.jp}\\

% ------------------------------------------------------------------------

\begin{thebibliography}{1}
  \bibitem{ZhuBMO} D.~B{\'e}koll{\'e}, C.~A.~Berger, L.~A.~Coburn, K.~H.~Zhu, {\it BMO in the Bergman metric on  bounded symmetric domains}, J. Funct. Anal. {\bf 93} (1990), no.2, 310-350.
  \bibitem{B-K}    D.~B{\'e}koll{\'e}, A.~T.~Kagou, {\it Reproducing properties and $L^p$-estimates for Bergman projections in Siegel domains of type II. }, Studia. Math. {\bf 115} (1995), 219-239. 
  \bibitem{BCZ}  C.~A.~Berger, L.~A.~Coburn, K.~H.~Zhu, {\it Function theory on Cartan domains and the Berezin-Toeplitz symbol calculus}, Amer. J. Math. {\bf 110} (1988), 921-953.
  \bibitem{Eng}  M.~Engli\v{s}, {\it Compact Toeplitz operators via the Berezin transform on bounded symmetric domains}, Integr. Equ. Oper. Theory {\bf 20} (1999), 426-455. 
  \bibitem{Hua} L.~K.~Hua, {\it Harmonic analysis of functions of several complex valiables on the classical domains}, Amer. Math. Soc., Providence, 1963 
  \bibitem{I-K}  H.~Ishi, C.~Kai, {\it The representative domain of a homogeneous bounded domain}, Kyushu J. Math. {\bf 64} (2010), 35-47. 
  \bibitem{I-Y}  H.~Ishi, S.~Yamaji, {\it Some estimates of the Bergman kernel of minimal bounded homogeneous domains}, preprint. 
  \bibitem{Koba} S.~Kobayashi, {\it Hyperbolic manifolds and holomorphic mappings (second edition) An introduction}, World Sci., 2005. 
  \bibitem{MM}  M.~Maschler, {\it Minimal domains and their Bergman kernel function}, Pacific J. Math. {\bf 6} (1956), 501-516.
  \bibitem{VGS}  {\`E}.~B.~Vinberg, S.~G.~Gindikin, I.~I.~Pjatecki{\u \i}-{\v S}apiro, {\it Classification and canonical realization of complex bounded homogeneous domains}, Trans. Moscow Math. Soc. {\bf 12} (1963), 404--437.
  \bibitem{Zhu1} K.~H.~Zhu, {\it Positive Toeplitz operators on weighted Bergman spaces of bounded symmetric domains}, J. Oper. Theor, {\bf 20} (1988), 329-357. 
  \bibitem{Zhu3} K.~H.~Zhu, {\it Operator Theory in Function Spaces, second edition}, Amer. Math. Soc., Mathematical Surveys and Monographs Vol.{\bf 138}, 2007.
\end{thebibliography}
\end{document}